\documentclass[11pt,a4paper]{article}
\usepackage{pifont}
\usepackage{bbding}

\usepackage{amssymb}
\usepackage{mathrsfs}
\usepackage{amsfonts}
\usepackage{amsthm,amscd,amsmath}
\usepackage{eufrak}
\usepackage{youngtab}
\usepackage{graphicx}
\usepackage[all]{xy}

\renewcommand{\atopwithdelims}[2]{%
\genfrac{[}{]}{0pt}{}{#1}{#2}}

\makeatletter
\def\@biblabel#1{#1}
\makeatother

\newtheorem{theorem}{Theorem}[section]
\newtheorem{lemma}[theorem]{Lemma}

\theoremstyle{definition}
\newtheorem{definition}[theorem]{Definition}

\theoremstyle{proposition}
\newtheorem{proposition}[theorem]{Proposition}

\theoremstyle{example}
\newtheorem{example}[theorem]{Example}

\theoremstyle{corollary}
\newtheorem{corollary}[theorem]{Corollary}

\date{}

\begin{document}

\title{Composition Series of Tensor Product}

\author{Bin Li and Hechun Zhang\\
\small{Department of Mathematical Science}, \\
\small{Tsinghua University, Beijing, P.R. China, 100084}\\
\scriptsize{Email: hzhang@math.tsinghua.edu.cn,
 libin07@mails.tsinghua.edu.cn}}

\maketitle

\begin{abstract}
Given a quantized enveloping algebra $U_q(\mathfrak g)$ and a pair
of dominant weights ($\lambda$, $\mu$), we extend a conjecture
raised by Lusztig in \cite{Lusztig:1992}
 to a more general form and then prove this extended
  Lusztig's conjecture. Namely we prove that for
any symmetrizable Kac-Moody algebra $\mathfrak g$, there is a
composition series of the $U_q(\mathfrak g)$-module
$V(\lambda)\otimes V(\mu)$ compatible with the canonical basis. As a
byproduct, the celebrated  Littlewood-Richardson rule is derived and
we also construct, in the same manner, a composition series of
$V(\lambda)\otimes V(-\mu)$ compatible with the canonical basis when
$\mathfrak g$ is of affine type and the level of $\lambda-\mu$ is
nonzero.
\end{abstract}


\section{Introduction}

Let $U_q(\mathfrak g)$ be a quantized enveloping algebra associated
to an arbitrary symmetrizable Kac-Moody algebra $\mathfrak g$.
  In \cite{Lusztig:1992}, for a pair of dominant integral functions
  ($\lambda,$ $\mu$), Lusztig constructed a canonical basis for the $U_q(\mathfrak g)$-module
   $V(\lambda)\otimes
   V(-\mu)$, where $V(\lambda)$ is an irreducible highest
   weight integrable $U_q(\mathfrak g)$-module with highest weight
   $\lambda$ and $V(-\mu)$ is an
   irreducible lowest weight integrable $U_q(\mathfrak g)$-module
   with lowest weight $-\mu$.
   This basis has many remarkable properties and can be lifted
   to a basis of the modified quantized enveloping algebra $\widetilde{U}$. Since then the
   canonical basis as well as the corresponding crystal basis of both this
   tensor product and $\widetilde{U}$
   are widely investigated by many mathematicians
   e.g. \cite{Beck:2004, Kashiwara:1994, Lusztig:1993, Lusztig:1995}.

  Due to the stable property of the basis, there are quite a few submodules of  $V(\lambda)\otimes
   V(-\mu)$ compatible with the canonical basis, that is, every such submodule
   is spanned by parts of the basis. Lusztig conjectured further in
  \cite{Lusztig:1992} that in the case $\mathfrak g$ is of finite
   type there is a composition series of $V(\lambda)\otimes
   V(-\mu)$ compatible with the canonical basis and he proved the
   conjecture in the case of type $A_1$ by a direct computation.
    Later in chapter 27 of \cite{Lusztig:1993}
   concerning about the based module,
   Lusztig proved that  for any integrable $U_q(\mathfrak g)$-module
    $M=\bigoplus_{\xi\in P_+}M[\xi]$ in category $\mathcal O_{int}$ where
    $M[\xi]$ is the sum of all submodules of $M$ isomorphic to $V(\xi)$,
    $M[\lambda]$ is compatible with the canonical basis of $M$ if
    $\lambda$ is maximal among those $\xi$ such that $M[\xi]$ is nonzero.
     Though not pointing out, Lusztig's proof of this result implies the
      conjecture and provided an inductive
   construction for the composition series since,  in particular,
   $V(\lambda)\otimes V(-\mu)$ is  in category $\mathcal O_{int}$ when
    $\mathfrak g$ is of finite type. The crystal structures of both
    $V(\lambda)\otimes V(-\mu)$ and $\widetilde{U}$ are extensively investigated
     by Kashiwara
     in \cite{Kashiwara:1994}. In \cite{Lusztig:1995} Lusztig investigated the two-sided cells in
      the canonical basis of $\widetilde{U}$ for $\mathfrak g$ of finite type
      and he raised some conjectures in affine type case which were finally
       solved by Beck and Nakajima in \cite{Beck:2004}.

   In \cite{FJKMMT:2004}, a filtration of $V(\Lambda_{i})\otimes
   V(-\Lambda_{j})$ of $U_q(\mathfrak g)$ was constructed, for $\mathfrak g$
    which is of affine type and
   where $\Lambda_i$ and $\Lambda_j$ are fundamental weights. Each
    $U_q(\mathfrak g)$-submodules appeared in this  filtration
  is generated by the tensor product of $u_{\Lambda_{i}}$
   with an extremal vector of $ V(-\Lambda_{j})$.  It turns out that
   all of the $U_q(\mathfrak g)$-submodules appeared in this  filtration
  are compatible with the canonical basis which can be proved using
  an important lemma
  due to Kashiwara and some results for Demazure modules.
  Motivated by the construction of the filtration in \cite{FJKMMT:2004},
  we construct the composition series of $V(\lambda)\otimes V(\mu)$
  directly for $\mathfrak g$ of
any type in the same fashion. The conjecture by Lusztig is then a
special case since $V(\mu)$ is also a lowest weight module for
$\mathfrak g$ of finite type. This is quite different from  the
argument in Chapter 27 in  Lusztig's book \cite{Lusztig:1993} and
one can derive from our proof the Littlewood-Richardson rule for
decomposing the tensor product
 $V(\lambda)\otimes V(\mu)$ into the direct sum of irreducible modules, which
 is also known by the work of Littelmann \cite{Littelmann:1994}.

On geometric aspects, quiver varieties were introduced by Nakajima
in order to get integrable highest weight representations of
symmetric Kac-Moody algebra $\mathfrak g$. Furthermore, there is
also a geometric construction of tensor product $V(\lambda_1)\otimes
\cdots\otimes V(\lambda_r)$ using quiver varieties
\cite{Nakajima:2001}. To realize this tensor product, Malkin also
introduced in \cite{Malkin:2003} the tensor product variety. Though
both constructions are in classical case ($q=1$), it would be
interesting to consider the geometric construction of the
composition series using Nakajima's quiver variety or Malkin's
tensor product variety. We will study this topic in the forth coming
publications.

   The arrangement of the paper is the following: in section 2, we
   recall some  basics of the theory of  crystal basis and canonical
   basis. In particular, we recall the construction of the canonical basis
   of  $V(\lambda)\otimes
   V(-\mu)$ due to Lusztig. Next in section 3, the extended Lusztig's
   conjecture is proved by building up the required composition series explicitly
   using the theory of crystal basis due to Kashiwara. Then we reintroduce
    the Littlewood-Richardson rule and
   compare this composition series with Lusztig's inductive construction.
   Finally in the last section we study the tensor product
    $V(\lambda)\otimes V(-\mu)$ for any symmetrizable
    Kac-Moody algebra $\mathfrak g$. In particular, the connected components of the crystal graph of
$U_q(g)a_{\lambda-\mu}$ are completely determined and a composition
series of $V(\lambda)\otimes V(-\mu)$ is constructed compatible with
the canonical basis when $\mathfrak g$ is of affine type and the
level of $\lambda-\mu$ is nonzero.

\section{Lusztig's Construction of Canonical Basis}

\subsection{Notations}

Let $\mathfrak g=\mathfrak g(A)$ be an arbitrary symmetrizable
Kac-Moody algebra over $\mathbb Q$ where $A$ is the $n\times n$
generalized Cartan matrix and let $\mathfrak h$ be the Cartan
subalgebra which is of dimension $2n-rank(A)$. We denote by $I=\{1,
\cdots, n\}$ the index set. Let $Q=\bigoplus_{i\in I}\mathbb Z\alpha_i$ be
the root lattice and set $Q_+=\bigoplus_{i\in I}\mathbb Z_+\alpha_i$ where
$\alpha_i$ are the simple roots. Denote by $\{h_i\in\mathfrak h\ |\
i\in I\}$ the set of simple coroots. $P^{\vee}$ is defined to be a
free $\mathbb Z$-module with a basis \[\{h_i\ |\ i\in I\}\bigcup
\{d_j\in\mathfrak h\ |\ 1\leqslant j\leqslant n-rank(A)\},\] called
the dual weight lattice. We also define $P=\{\lambda\in \mathfrak
h^{\ast}\ |\ \langle h, \lambda\rangle\in \mathbb Z\ \forall h\in
P^{\vee}\}$ to be the weight lattice. Note that there is a symmetric
bilinear form on $P$ such that $\frac{2(\alpha_i,
\lambda)}{(\alpha_i,\alpha_i)}=\langle h_i, \lambda\rangle$ for
$i\in I$, $\lambda\in P$. Let $P_+=\{\lambda\in \mathfrak
h^{\ast}|\langle h_{i},\lambda\rangle\in \mathbb Z_+\ \ \forall i\in I\
\}$ be the set of dominant weights. Denote by $\Lambda_i$ the
fundamental weight, i.e. $\langle h_i, \Lambda_j\rangle=\delta_{ij}\
\ \forall i,j\in I.$ The partial order on $P$ is defined as
$\xi\geqslant\varphi$ if
$\xi-\varphi\in Q_+$.\\
\indent The quantized enveloping algebra $U_q(\mathfrak g)$ is
defined as a $k$-algebra with generators $E_i$, $F_i$ and $q^{h}$
for all $i\in I$ and $h\in P^{\vee}$, where $k=\mathbb Q(q)$. The
relations are as in \cite{Kashiwara:1994}. Let $U_q(\mathfrak g)^+$
$(resp.$ $U_q(\mathfrak g)^-$) be the subalgebra of $U_q(\mathfrak
g)$ generated by the $E_i$ ($resp.$ $F_i$) for all $i\in I$. Note
that irreducible integrable highest and lowest weight $U_q(\mathfrak
g)$-modules can be indexed by $P_+$ and $-P_+$ respectively. Namely,
for $\lambda\in P_+$ ($resp.\
 \lambda\in -P_+)$, we denote by $V(\lambda)$ the irreducible highest
($resp.$ lowest) weight $U_q(\mathfrak g)$-module with highest
($resp.$ lowest) weight $\lambda$ and let $u_{\lambda}$ be the
highest ($resp.$ lowest) weight vector. Let $\mathcal {O}_{int}$
denote the category of integrable $U_q(\mathfrak g)$-modules $M$
which are direct sums of irreducible integrable highest weight
modules.

As is widely known, if $\mathfrak g$ is of finite type, the Weyl
group $W$ of the Lie algebra $\mathfrak g$ is a finite group and
there is a unique longest element $w_0\in W$. In this case, the
irreducible module $V(\lambda)$ is finite dimensional  and hence it
is also a  lowest weight module with the lowest weight $w_0\lambda$.

Note that $U_q(\mathfrak g)$ is a Hopf algebra and thus the tensor
product of $U_q(\mathfrak g)$-modules has a structure of
$U_q(\mathfrak g)$-module through the coproduct on $U_q(\mathfrak
g)$. There is a $\mathbb Q$-automorphism of $U_q(\mathfrak g)$, denoted
by $^-$, such that
\[\overline{q}=q^{-1},\ \overline{q^{h}}=q^{-h},\
\overline{E_i}=E_i, \ \overline{F_i}=F_i .\]

Let $\widetilde{U}_q(\mathfrak g)$ or simply
 $\widetilde{U}$ be the modified quantized enveloping algebra \cite{Kashiwara:1994}
 generated by $U_q(\mathfrak g)a_{\lambda}$
 for $\lambda\in P$ subject to the relations:
\[q^{h}a_{\lambda}=q^{\langle h, \lambda\rangle}a_{\lambda},\ a_{\lambda}a_{\mu}=\delta_{\lambda,\mu}a_{\lambda},\
ua_{\lambda}=a_{\lambda+\xi}u\ \textrm{for}\ u\in U_q(\mathfrak
g)_{\xi}\] where $U_q(\mathfrak g)_{\xi}=\{u\in U_q(\mathfrak g)\ |\
q^{h}uq^{-h}=q^{\langle h, \lambda\rangle}u \ \forall h\in P^{\vee}\}.$ Note that
$\widetilde{U}=\bigoplus_{\lambda\in P}U_q(\mathfrak g)a_{\lambda}.$

\subsection{Canonical Basis}
Canonical bases are constructed by  Lusztig for both $U_q(\mathfrak
g)^\pm$ and some kinds of $U_q(\mathfrak g)$-modules
\cite{Lusztig:1990, Lusztig2:1990, Lusztig:1991, Lusztig:1992}. This
basis was subsequently studied by M.Kashiwara \cite{Kashiwara:1990,
Kashiwara:1991, Kashiwara2:1993, Kashiwara:1994} who called it the
global crystal basis. Hereafter we will follow Lusztig's terminology
of \textit{canonical basis} while using the notations of global
crystal basis due to Kashiwara.

 For details on definition of (abstract) crystal, one can refer to
 \cite{Kashiwara:1993}. We only mention here that for $\lambda\in P_+$,
 $V(\lambda)$ admits a crystal basis $(L(\lambda), B(\lambda))$ where
$B(\lambda)=\{\tilde{f}_{i_1}\cdots\tilde{f}_{i_r}u_{\lambda}+qL(\lambda)\in
L(\lambda)/qL(\lambda)\ |\ r\geqslant 0, i_k\in I\ \}\setminus\{0\}$
and there is a similar result for lowest weight module $V(-\lambda)$
\cite{Kashiwara:1990, Kashiwara:1991}. We denote also by
$u_{\lambda}$ its image in $L(\lambda)/qL(\lambda)$ if this causes
no confusion. For a $U_q(\mathfrak g)$-module $M$, there is an
involution $^-$ on $M$ such that
\[\overline{u\cdot m}=\overline{u}\cdot\overline{m} \ \ \forall u\in
U_q(\mathfrak g),\ m\in M,\] which will be called bar involution
hereafter. Suppose that there is a balance triple ($L(M),
\overline{L(M)}, M_{\mathbb Q}$) for $M$, then we have a basis
consisting of bar-invariant elements, called \textit{canonical
basis} in this paper (see \cite{Kashiwara:1991}
 for details). It is denoted
by $\{G(b)| b\in B(M)\}$ where $(L(M), B(M))$ forms the crystal
basis of $M$.
\begin{definition}
Let $M$ and $N$ be $U_q(\mathfrak g)$-modules with canonical bases,
\begin{itemize}
 \item[(i)]
  a $U_q(\mathfrak g)$ (or $U_q(\mathfrak g)^{\pm}$)-submodule
$M^{\prime}$ of $M$ is said to be nice (or compatible with the
canonical basis of $M$) if $M^{\prime}$ is spanned as a $k$-vector
space by parts of the canonical basis of $M$.
\item[(ii)]
  a $U_q(\mathfrak g)$-morphism $f:\
M\longrightarrow N $ is said to be nice (or compatible with
canonical bases) if $f$ maps any canonical basis element of $M$ to
either zero or a canonical basis element of $N$ and if $kerf$ is
nice.
\item[(iii)]
   a filtration or a composition series of a $U_q(\mathfrak
g)$-module $M$ is said to be nice (or compatible with the canonical
basis) if any submodule in the filtration or composition series is
nice.
\end{itemize}
\end{definition}

For $\lambda\in \pm P_+$, we define the bar involution on
$V(\lambda)$ by \[\overline{x\cdot u_{\lambda}}=\overline{x}\cdot
u_{\lambda}\] for all $x\in U_q(\mathfrak g)$. As is known to all,
$V(\lambda)$ has a canonical basis $\{G(b)|b\in B(\lambda)\}$. Note
that $U_q(\mathfrak g)^\mp$ also has a canonical basis $\{G(b)|b\in
B(\pm\infty)\}$ such that $\{G(b)u_{\lambda}|b\in
B(\pm\infty)\}\backslash\{0\}$ coincides with the above set.

\subsection{Canonical Bases in Tensor Product}

For $U_q(\mathfrak g)$-modules $M$ and $N$ with bar involutions
where $M\in\mathcal {O}_{int}$, the $U_q(\mathfrak g)$-module
$M\otimes N$ can be endowed with a bar involution as
\[\overline{u\otimes v}=\Theta(\overline{u}\otimes \overline{v})\] for
all $u\in M, v\in N$, where $\Theta$ is the quasi R-matrix
\cite{Jantzen:1996}.

 We focus our attention on
$V(\lambda)\otimes V(\mu)$, where $\lambda,\ \mu\in P_+$. Since both
$V(\lambda)$ and $V(\mu)$ have canonical bases, $V(\lambda)\otimes
V(\mu)$ has a natural basis $\{G(b_1)\otimes G(b_2)| b_1\in
B(\lambda),b_2\in B(\mu) \}.$ The bar involution acts on this basis
as
\[\overline{G(b_1)\otimes G(b_2)}\in G(b_1)\otimes
G(b_2)+\sum_{wtb_1^{\prime}>wtb_1,wtb_2^{\prime}<wtb_2}\mathbb
Z[q,q^{-1}]G(b_1^{\prime})\otimes G(b_2^{\prime}).\] Thus we get a
new basis that is bar-invariant with upper triangular relations with
the above natural one.
\begin{proposition} (\cite{Lusztig:1992})
For $b_1\otimes b_2\in B(\lambda)\otimes B(\mu)$ there exists a
unique element \[(b_1\diamond b_2)_{\lambda,\mu} \in G(b_1)\otimes
G(b_2)+\sum_{wtb_1^{\prime}>wtb_1,wtb_2^{\prime}<wtb_2}q\mathbb
Z[q]G(b_1^{\prime})\otimes G(b_2^{\prime}) \] satisfying
$\overline{(b_1\diamond b_2)_{\lambda,\mu}}=(b_1\diamond
b_2)_{\lambda,\mu}$. Hence $\{(b_1\diamond b_2)_{\lambda,\mu}|
b_1\in B(\lambda),b_2\in B(\mu)\}$ forms a new basis of
$V(\lambda)\otimes V(\mu)$.
\end{proposition}

Note that $V(\lambda)\otimes V(\mu)$ has a crystal basis
$(L(\lambda)\otimes L(\mu)$, $B(\lambda)\otimes B(\mu))$ and for
$b_1\otimes b_2\in B(\lambda)\otimes B(\mu)$, the corresponding
canonical basis element $G(b_1\otimes b_2)=(b_1\diamond
b_2)_{\lambda,\mu}$. In particular, $G(b_1\otimes b_2)=G(b_1)\otimes
G(b_2)$ if $b_1=u_{\lambda}$. This basis is constructed in the same
fashion as that of Lusztig's canonical basis of $V(\lambda)\otimes
V(-\mu)$ \cite{Lusztig:1992}. When $\mathfrak g$ is of finite
type, our basis coincides with Lusztig's basis for
$V(\lambda)\otimes V(w_0\mu)$ since the $U_q(\mathfrak g)$-morphism
$f:V(\mu)\longrightarrow V(w_0\mu)$ which takes $u_{\mu}$ to the
canonical basis element of hight weight in $V(w_0\mu)$ is easily
seen to be a nice isomorphism. Therefore $V(\lambda)\otimes V(-\mu)$
is a special case in our consideration for $\mathfrak g$ of finite
type but things are quite different in affine or indefinite types
since this tensor product is not in category $\mathcal O_{int}$ any
more. As is known $V(\lambda)\otimes V(-\mu)$ is a cyclic
$U_q(\mathfrak g)$-module generated by $u_{\lambda}\otimes
u_{-\mu}$. We mention here a result of Lusztig's (Theorem 2 in
\cite{Lusztig:1992}) on the stability property for the canonical
basis of this tensor product, which is actually true for $\mathfrak
g$ of any type.
\begin{proposition}\label{stability}
 For any $\lambda,\mu,\theta\in
P_+$, the $U_q(\mathfrak g)$-morphism \[\phi:
V(\lambda+\theta)\otimes V(-\theta-\mu) \longrightarrow
V(\lambda)\otimes V(-\mu)\] which takes $u_{\lambda+\theta}\otimes
u_{-\theta-\mu}$ to $u_{\lambda}\otimes u_{-\mu}$ is a surjective
nice $U_q(\mathfrak g)$-morphism.
\end{proposition}

We can get some submodules of $V(\lambda)\otimes V(-\mu)$ compatible
with the canonical basis of $V(\lambda)\otimes V(-\mu)$ by means of
the above maps, but usually one cannot get a composition series
consisting of the nice submodules obtained above.

\begin{example}\label{nocom}
 In $A_2$ case, consider $V(\Lambda_1)\otimes
V(-\Lambda_1-\Lambda_2)$. Since we have \[V(\Lambda_1)\otimes
V(-\Lambda_1-\Lambda_2) \xrightarrow{\ \ \phi\ \ } V(0)\otimes
V(-\Lambda_2)\cong V(-\Lambda_2)\] then $V(\Lambda_1)\otimes
V(-\Lambda_1-\Lambda_2)\supseteq ker\phi \supseteq 0$ is a
filtration compatible with the canonical basis, but $ker\phi$ is far
from being an irreducible module.
\end{example}

We denote by $B(\lambda,-\mu)$ the crystal basis of
$V(\lambda)\otimes V(-\mu)$. It can be seen from
Proposition~\ref{stability} that there is an embedding of crystals
$B(\lambda,-\mu)\hookrightarrow B(\lambda+\theta,-\theta-\mu)$ and
note that it is strict. For $\lambda,\mu\in P_+$, let $\Phi:
U_q(\mathfrak g)a_{\lambda-\mu}\longrightarrow V(\lambda,-\mu)$ be
the $U_q(\mathfrak g)$-map taking $a_{\lambda-\mu}$ to
$u_{\lambda}\otimes u_{-\mu}$. It is known that $\widetilde{U}$ as
well as each $U_q(\mathfrak g)a_{\lambda}$ have canonical bases and
$\Phi$ is a nice surjective $U_q(\mathfrak g)$-map. We denote the
crystal basis of $\widetilde{U}$ ($resp$. $U_q(\mathfrak
g)a_{\lambda}$) by $\widetilde{B}$ ($resp$. $B(U_q(\mathfrak
g)a_{\lambda})$). Hence we have an embedding of crystals
$B(\lambda,-\mu)\hookrightarrow
 B(U_q(\mathfrak g)a_{\lambda-\mu}).$ It can be viewed as
 \[B(\lambda,-\mu)\subseteq B(\lambda+\theta,-\theta-\mu)
 \subseteq B(U_q(\mathfrak g)a_{\lambda-\mu})
 \subseteq \widetilde{B}.\] Note that \[B(U_q(\mathfrak g)a_{\lambda})\cong B(\infty)
 \otimes T_{\lambda}\otimes B(-\infty)\] where $T_{\lambda}$ is a crystal
 consisting of a single element $t_{\lambda}$ with
 $\varepsilon_i(t_{\lambda})=\varphi_i(t_{\lambda})=-\infty$
 for all $i\in I$.
 For $b\in B(\lambda,-\mu)\subseteq\widetilde{B}$, we denote
 the corresponding canonical basis element in $V(\lambda,-\mu)$ or $\tilde{U}$
 by the same $G(b)$ if there is no confusion.

\section{Composition Series of $V(\lambda)\otimes V(\mu)$}

\subsection{Kashiwara's Lemma}

We fix $\lambda, \mu\in P_+$ hereafter. In \cite{Lusztig:1992},
Lusztig conjectured that there exists a nice composition series of
$V(\lambda)\otimes V(-\mu)$ if $\mathfrak g$ is of finite type. One
may extend this conjecture by changing $V(-\mu)$ to $V(\mu)$ and
omitting the assumption that $\mathfrak g$ is of finite type. This
section is devoted to the proof of this extended Lusztig's
conjecture. In order to do that, we need the following lemma due to
Kashiwara \cite{Kashiwara:1993} who proved the lemma in case of
$g=sl_2$ and claimed that it is true in general.
\begin{lemma}\label{kashiwaralemma}
(\cite{Kashiwara:1993})  Let $M$ be an integrable $U_q(\mathfrak
g)$-module with a canonical basis. If $N$ is a nice $U_q(\mathfrak
g)^+$-submodule of $M$, then $U_q(\mathfrak g)N$ is a nice
$U_q(\mathfrak g)$-submodule of $M$, i.e. $U_q(\mathfrak
g)N=\bigoplus_{b\in B(U_q(\mathfrak g)N)\subseteq B(M)} kG(b).$
Moreover, $B(U_q(\mathfrak g)N)=\{
\tilde{f}_{i_1}\cdots\tilde{f}_{i_m}b\ |\ m\geqslant 0,
i_1,\cdots,i_m\in I,b\in B(N) \}\setminus \{0\}$.
\end{lemma}

For completeness, we give a full proof of Kashiwara's lemma. First
assume that $M$ is a finite dimensional $U_q(sl_2)$-module with
canonical basis and we denote by $B(M)$ or $B$ for simplicity the
crystal basis of $M$. As is defined by M. Kashiwara in
\cite{Kashiwara:1993}, $I^{l}(M)$ is the sum of all
$l+1$-dimensional irreducible submodules of $M$. Hence $M=\bigoplus _l
I^l(M).$ Set $I^l(B)=\{b\in B| \varepsilon(b)+\varphi(b)=l\}$ and one can
see that $B=\bigoplus _l I^l(B)$ where $\bigoplus$ here simply means a
union. Note that the decomposition of $M$ into isotypical components
$I^l(M)$'s is compatible with the decomposition of crystal basis $B$
into $I^l(B)$'s, but it is usually not compatible with the canonical
 basis. Set $W^l(M)=\bigoplus_{l^{\prime}\geqslant l}I^{l^{\prime}}(M)$
and $W^l(B)=\{b\in B| \varepsilon(b)+\varphi(b)\geqslant l\}.$ We know
from \cite{Kashiwara:1993} that $W^l(M)$ is a nice
$U_q(sl_2)$-submodule of $M$, i.e. $W^l(M)=\bigoplus_{b\in W^l(B)}
kG(b).$ Moreover, if $b\in I^l(B)$, then
\[F_i^{(k)}G(b)={\atopwithdelims{\varepsilon_i(b)+k}{k}}_iG(\tilde{f}_i^kb)\ \ \ \ mod\ W^{l+1}(M),\]
\[E_i^{(k)}G(b)={\atopwithdelims{\varphi_i(b)+k}{k}}_iG(\tilde{e}_i^kb)\ \ \ \ mod\ W^{l+1}(M).\]
Let $N$ be a nice $U_q(sl_2)^+$-submodule of $M$, i.e.
$N=\bigoplus_{b\in B(N)\subseteq B(M)}kG(b).$ Set
$\widetilde{N}=U_q(sl_2)N$, $I^l(B(N))=B(N)\bigcap I^l(B)$,
$W^l(B(N))=\bigcup_{k\geqslant l}I^k(B(N))$, $W^l(N)=W^l(M)\bigcap
N$ and $B(\widetilde{N})=\bigcup_{m\geqslant 0}
\tilde{f}^mB(N)\setminus \{0\}$. We have the following lemma.
\begin{lemma} (\cite{Kashiwara:1993}) \label{sl_2}For $N$, $W^{l}(N)$,
$\widetilde{N}$, $B(\widetilde{N})$ defined as above,
\begin{itemize}
\item[(i)]
  $\tilde{e}_iB(N)\subseteq B(N)\bigcup \{0\}.$
\item[(ii)]
$W^l(N)=\bigoplus_{b\in W^l(B(N))}kG(b).$
\item[(iii)]
$W^l(\widetilde{N})=U_q(sl_2)W^l(N).$
\item[(iv)]
$\widetilde{N}=\bigoplus_{b\in B(\widetilde{N})\subseteq B(M)}
kG(b).$\end{itemize}
\end{lemma}
\begin{definition}
An integrable $U_q(sl_2)$-module $M$ is said to be truncated if
$M=\bigoplus_{j\geqslant 0}I^{j}(M)$ where there exists an $l\geqslant
0$ such that $I^{j}(M)=0$ for all $j\geqslant l$.
\end{definition}
 Recall that Lemma~\ref{sl_2} (\romannumeral 4) is proved by
showing  \[W^l(\widetilde{N})=\bigoplus_{b\in
W^l(B(\widetilde{N}))}kG(b)\] through a descending induction
 on $l$ since both of the two sides equal zero
 when $l$ is sufficiently large. Thus the above results also hold when we modify
 $M$ to be a truncated integrable $U_q(sl_2)$-module, that is,
\begin{lemma} \label{truncated}
Let $M$ be a truncated integrable $U_q(sl_2)$-module with a
canonical basis. If $N$ is a nice $U_q(sl_2)^+$-submodule of $M$,
then $U_q(sl_2)N$ is a nice $U_q(sl_2)$-submodule of $M$, i.e.
$U_q(sl_2)N=\bigoplus_{b\in B(U_q(sl_2)N)\subseteq B(M)} kG(b).$
Moreover, $B(U_q(sl_2)N)=\bigcup_{m\geqslant 0}
\tilde{f}^mB(N)\setminus \{0\}$.
\end{lemma}
Furthermore, we can prove the following lemma.
\begin{lemma}\label{generalsl_2}
Let $M$ be an (possibly infinite dimensional) integrable
$U_q(sl_2)$-module with a canonical basis. If $N$ is a nice
$U_q(sl_2)^+$-submodule of $M$, then $U_q(sl_2)N=U_q(sl_2)^-N$ is a
nice $U_q(sl_2)$-submodule of $M$. Moreover,
$B(U_q(sl_2)N)=\bigcup_{m\geqslant 0} \tilde{f}^mB(N)\setminus
\{0\}$.
\end{lemma}
\begin{proof}
 One can define a nice $U_q(sl_2)$-submodule $W^l(M)$ of $M$ for
any $l\geqslant 0$ as before. Hence $M/W^l(M)$ is a truncated module
with a canonical basis $\{G(b)+W^l(M)| b\in I^j(B), j<l\}$ and
$(N+W^l(M))/W^l(M)$ is a nice $U_q(sl_2)^+$-submodule. Applying
Lemma~\ref{truncated}, we have
\[U_q(sl_2)\frac{N+W^l(M)}{W^l(M)}=\bigoplus_{b\in \bigoplus_{j<l}I^{j}(B(N))}
k(G(\tilde{f}^mb)+W^l(M)).\] It follows that
\[U_q(sl_2)(N+W^l(M))=(\bigoplus_{b\in
\bigoplus_{j<l}I^{j}(B(N))}kG(\tilde{f}^mb))\bigoplus(\bigoplus_{b\in
\bigoplus_{j\geqslant l}I^j(B)}kG(b)).\] Set
$\widetilde{N}=U_q(sl_2)N$. We have
$U_q(sl_2)(N+W^l(M))=\widetilde{N}+W^l(M).$ Hence
\[\widetilde{N}=\bigcap_{l\geqslant
0}(\widetilde{N}+W^l(M))=\bigcap_{l\geqslant 0}(\bigoplus_{b\in
\bigoplus_{j<l}I^{j}(B(N))}kG(\tilde{f}^mb)\bigoplus\bigoplus_{b\in
\bigoplus_{j\geqslant l}I^j(B)}kG(b))\] which is easily seen to be a
nice $U_q(sl_2)$-submodule of $M$. We denote by $B^l$ the
crystal basis of $\widetilde{N}+W^l(M)$, i.e. \[B^l=\{\tilde{f}^mb\
|\ b\in I^j(B(N)), \ j<l,\ m\geqslant 0,\ \tilde{f}^mb\neq 0\}\cup
W^l(B).\] Since $\tilde{f}^{m}b\in I^j(B)$ for $b\in I^j(B(N))$ and
$m\geqslant 0$ such that $\tilde{f}^mb\neq 0$, we have for $l< k$,
$B^l\supseteq B^k$ and $B^k\bigcap I^l(B)= \bigcup_{m\geqslant
0}\tilde{f}^mI^l(B(N))\setminus\{0\}$. It follows that
\[B(\widetilde{N})\bigcap I^l(B)=(\bigcap_{k\geqslant 0}B^k)\bigcap I^l(B)
=\bigcup_{m\geqslant 0}\tilde{f}^mI^l(B(N))\setminus\{0\}\] and
hence we have $B(\widetilde{N})=\bigcup_{l\geqslant
0}(B(\widetilde{N})\bigcap I^l(B)) =\bigcup_{m\geqslant 0}
\tilde{f}^mB(N)\setminus\{0\}$.
\end{proof}

We define $U_q(sl_2(i))$ to be the subalgebra of $U_q(\mathfrak g)$
generated by $E_i$, $F_i$ and $q^{\frac{(\alpha_i,\alpha_i)}{2}h_i}$
for some $i\in I$. Since $N$ is a nice $U_q(\mathfrak
g)^+$-submodule of $M$, it is also a nice
$U_q(sl_2(i))^+$-submodule. Hence $U_q(sl_2(i))N$ is a nice
$U_q(sl_2(i))$-submodule of $M$ by Lemma~\ref{generalsl_2}. It is
easy to see that $U_q(\mathfrak
g)^+U_q(sl_2(i))=U_q(sl_2(i))U_q(\mathfrak g)^+.$ Hence
\[U_q(sl_2(i))N=U_q(sl_2(i))U_q(\mathfrak g)^+N=U_q(\mathfrak g)^+U_q(sl_2(i))N\]
is still a $U_q(\mathfrak g)^+$-module. Repeating this, one can see
that \[U_q(sl_2(i_1))\cdots U_q(sl_2(i_m))N\] is a nice
$U_q(\mathfrak g)^+$-submodule of $M$ which admits a crystal basis
$\{\tilde{f}_{i_1}^{r_1}\cdots\tilde{f}_{i_m}^{r_m}b\ |$
$\ r_1,\
\cdots, r_m\in \mathbb Z_+,\ b\in B(N)\}\setminus \{0\}.$ This proves
Lemma~\ref{kashiwaralemma} since \[U_q(\mathfrak
g)N=\sum_{i_1,\cdots,i_m\in I}U(sl_2(i_1))\cdots U_q(sl_2(i_m))N.\]

\subsection{Composition Series}

The following construction of composition series is inspired by
\cite{FJKMMT:2004}. For $b\in B(\mu)$ with $wtb=\mu-\sum_{i\in
I}m_i\alpha_i$ where $m_i\geqslant0$, set $l(b)=\sum_{i\in I}m_i$.
Since $B(\mu)=\{\tilde{f}_{i_1}\cdots\tilde{f}_{i_l}u_{\mu}\ |\ i_1,
\cdots, i_l\in I, l\geqslant 0 \}\setminus \{0\},$ b is of the form
$\tilde{f}_{i_1}\cdots\tilde{f}_{i_l}u_{\mu}$ for some $i_1, \cdots,
i_l\in I$, $l\geqslant 0$. Hence
$wtb=\mu-\sum_{j=1}^{l}\alpha_{i_j}$, which implies $l=l(b)$. One
can define $|b|$ to be the $l(b)$-tuple $(i_1, \cdots, i_{l(b)})$
such that $(i_1, \cdots, i_{l(b)})$ is minimal in lexicographic
order among tuples $(j_1, \cdots,
 j_{l(b)})$ such that
 $\tilde{f}_{j_1}\cdots\tilde{f}_{j_{l(b)}}u_{\mu}=b$,
 i.e. \[|b|=min\{ (j_1, \cdots, j_{l(b)})\ |\ b=\tilde{f}_{j_1}
 \cdots\tilde{f}_{j_{l(b)}}u_{\mu} \}.\]
 Set $|u_{\mu}|=0$. Note that the order on $I$ is given as
 $1 < 2 < \cdots < n-1 < n .$
 If $|b_1|=|b_2|=(i_1, \cdots, i_l)$, we have
$b_1=b_2=\tilde{f}_{i_1}\cdots\tilde{f}_{i_l}u_{\mu}$
 which implies that there is a one to one correspondence between
 $B(\mu)$ and $\{|b|\ |\ b\in B(\mu)\}$. Thus we have a
 total order on $B(\mu)$ as the following:\\
 \[b_1\leqslant b_2\ \ \textrm{iff} \ \ l(b_1)>l(b_2)\ \  \textrm{or}\ \ l(b_1)=l(b_2)\ \
 \textrm{but} \ \ |b_1|\geqslant|b_2|. \]
 Obviously $b_1<b_2$ if $wtb_1<wtb_2$.
 \begin{example}\label{ordereg}In the case of type $A$, there is a combinatorial
 realization of the crystal $B(\lambda)$ for
 $\lambda\in P_+$. If $U_q(\mathfrak g)=U_q(sl_3)$,
$B(\Lambda_1+\Lambda_2)\cong B(\small\Yvcentermath1\yng(2,1))$ and
the crystal graph is given as the following
\[\includegraphics[width=54mm]{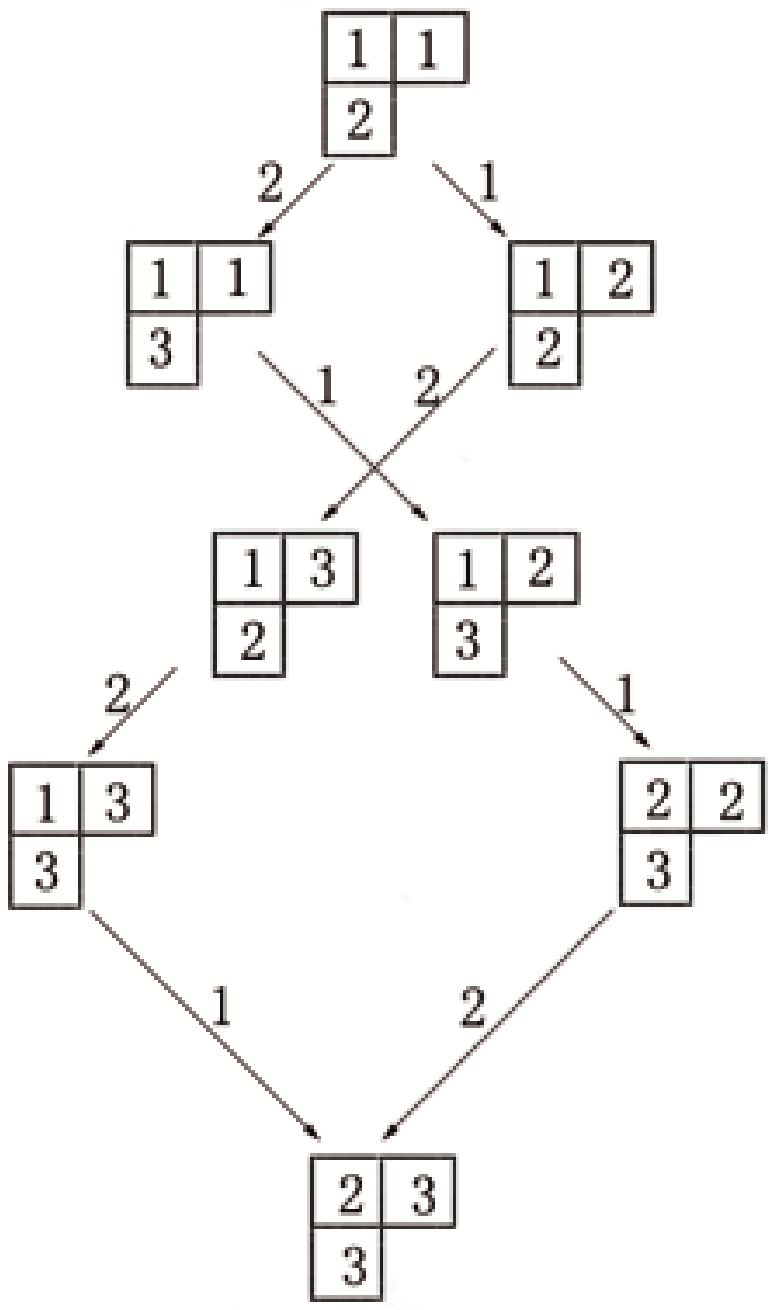} \] We have $|\small\Yvcentermath1\young(11,2)|=0,\
|\small\Yvcentermath1\young(11,3)|=(2),\
|\small\Yvcentermath1\young(12,2)|=(1),\
|\small\Yvcentermath1\young(13,2)|=(2,1),\
|\small\Yvcentermath1\young(12,3)|=(1,2),\
|\small\Yvcentermath1\young(13,3)|=(2,2,1),\
|\small\Yvcentermath1\young(22,3)|=(1,1,2),\
|\small\Yvcentermath1\young(23,3)|=(1,2,2,1).$
 Hence the order on $B(\Lambda_1+\Lambda_2)$ is given as the following
\[\small\Yvcentermath1\young(11,2)>\small\Yvcentermath1\young(12,2)>
\small\Yvcentermath1\young(11,3)> \small\Yvcentermath1\young(12,3)>
\small\Yvcentermath1\young(13,2)> \small\Yvcentermath1\young(22,3)>
\small\Yvcentermath1\young(13,3)>
\small\Yvcentermath1\young(23,3).\]
\end{example}

For $b\in B(\mu)$, we define a $k$-subspace $V_b(\mu)$ of $V(\mu)$
spanned by all canonical basis elements $G(c)$ such that $c\geqslant
b$, i.e. $V_b(\mu):= \sum_{c\geqslant b}kG(c).$

\begin{lemma}
For $\mu\in P_+$ and $b\in B(\mu)$, $V_b(\mu)$ is
 a nice $U_q(\mathfrak g)^+$-submodule of $V(\mu)$ and
 $B(V_b(\mu))=\{c\in B(\mu)\ |\ c\geqslant b\}$.
\end{lemma}

\begin{proof}
We only need to show that $V_b(\mu)$ is a
 $U_q(\mathfrak g)^+$-submodule of $V(\mu)$. For any $c\in B(\mu)$
 where $c\geqslant b$, one can see that $V_{c}(\mu)\subseteq V_b(\mu)$
 and $U_q(\mathfrak g)^+G(c)=\bigoplus_{\xi\in Q_+}U_q(\mathfrak g)_{\xi}^+G(c)
 =kG(c)\bigoplus\bigoplus_{\xi\in Q_+\setminus\{0\}}U_q(\mathfrak
 g)_{\xi}^+G(c).$
 For $\xi\in Q_+\setminus\{0\}$,
 \[U_q(\mathfrak g)_{\xi}^+G(c)\subseteq V(\mu)_{wtc+\xi}=
 \sum_{wtd=wtc+\xi}kG(d)\]
 \[\subseteq \sum_{wtd>wtc}kG(d)\subseteq\sum_{d\geqslant
 c}kG(d)=V_{c}(\mu).\]
 Hence $\bigoplus_{\xi\in Q_+\setminus\{0\}}U_q(\mathfrak g)_{\xi}^+G(c)\subseteq
 V_{c}(\mu)$ and furthermore,
 $U_q(\mathfrak g)^+G(c)\subseteq V_{c}(\mu)\subseteq V_b(\mu).$
 It follows that $U_q(\mathfrak g)^+V_b(\mu)=\sum_{c\geqslant b}
 U_q(\mathfrak g)^+G(c)\subseteq V_b(\mu).$ Thus $V_b(\mu)$ is a
 nice $U_q(\mathfrak g)^+$-submodule of $V(\mu)$.
 \end{proof}

Clearly,  the above proof is independent of the order on $B(\mu)_l
=\{ b\in B(\mu)\ |\ l(b)=l\}$. More generally, we can choose any
total order on $B(\mu)$ such that $b_1<b_2$ if $wtb_1<wtb_2$.

 For $b\in B(\mu)$, we define a
  $U_q(\mathfrak g)$-submodule $F_{\lambda}(b)$ of $V(\lambda)\otimes V(\mu)$ generated
  by $u_{\lambda}\otimes V_b(\mu)$, i.e.
  \[F_{\lambda}(b):= U_q(\mathfrak g)(u_{\lambda}\otimes V_b(\mu)).\]
  Since it follows from the coproduct formula that
\[U_q(\mathfrak g)^+(u_{\lambda}\otimes V_b(\mu))=
  u_{\lambda}\otimes U_q(\mathfrak g)^+V_b(\mu)=u_{\lambda}\otimes
  V_b(\mu)\] and \[u_{\lambda}\otimes
  V_b(\mu)=\sum_{c\geqslant b}ku_{\lambda}\otimes G(c)=
  \sum_{c\geqslant b}kG(u_{\lambda}\otimes c),\] $u_{\lambda}\otimes V_b(\mu)$
is a nice $U_q(\mathfrak g)^+$-submodule of $V(\lambda)\otimes
V(\mu)$. We have the following proposition according to
Lemma~\ref{kashiwaralemma}.

\begin{proposition}\label{submodule}
For $\lambda$, $\mu\in P_+$ and $b\in B(\mu)$, $F_{\lambda}(b)$ is a
nice $U_q(\mathfrak g)$-submodule of $V(\lambda)\otimes V(\mu)$.
Moreover, $B(F_{\lambda}(b))=\{\
\tilde{f}_{i_1}\cdots\tilde{f}_{i_l}(u_{\lambda}\otimes c)\ |\ i_1,
\cdots,$ $i_l\in I, l\geqslant0, c\geqslant b \}\setminus\{0\}.$
\end{proposition}

\begin{theorem}\label{main}
For $\lambda$, $\mu\in P_+$, $\{\ F_{\lambda}(b)\ |\ b\in B(\mu)\}$
forms a nice ascending filtration of $V(\lambda)\otimes V(\mu)$ as
the following
\[0\subseteq F_{\lambda}(b_1)\subseteq F_{\lambda}(b_2)
\subseteq F_{\lambda}(b_3)\subseteq\cdots\ \ \ \ \ (3.1)\] where
$u_{\mu}=b_1>b_2>b_3>\cdots$ is a complete list of $B(\mu)$.
Moreover, for two neighbors $c>b$ in $B(\mu)$,
$F_{\lambda}(b)/F_{\lambda}(c)\cong V(\lambda+wtb)$ if
$\tilde{e}_i(u_{\lambda}\otimes b)=0$ for all $i\in I$, otherwise
$F_{\lambda}(b)=F_{\lambda}(c)$.
\end{theorem}

\begin{proof}
It suffices to show the second half. We have
$B(F_{\lambda}(b))\supseteq B(F_{\lambda}(c))$ if $c>b$ are two
neighbors in $B(\mu)$. Claim that \[B(F_{\lambda}(b))\setminus
B(F_{\lambda}(c))=
\{\tilde{f}_{i_1}\cdots\tilde{f}_{i_l}(u_{\lambda}\otimes b)\ |\
i_1, \cdots, i_l\in I, l\geqslant0\}\setminus\{0\}\] if
$\tilde{e}_i(u_{\lambda}\otimes b)=0$ for all $i\in I$, otherwise
$B(F_{\lambda}(b))=B(F_{\lambda}(c))$. Indeed, if
$B(F_{\lambda}(b))\setminus B(F_{\lambda}(c))$ is non-empty, it
follows from Proposition~\ref{submodule} that any element in
$B(F_{\lambda}(b))\setminus B(F_{\lambda}(c))$ is of the form
$\tilde{f}_{j_1}\cdots\tilde{f}_{j_k}(u_{\lambda}\otimes d)$ for
some $j_1, \cdots, j_k\in I$, $k\geqslant0$ and $d\in B(\mu)$ where
$c>d\geqslant b$ and it implies $d=b$. Hence if $u_{\lambda}\otimes
b\in B(F_{\lambda}(b))\setminus B(F_{\lambda}(c))$, we have 
\[B(F_{\lambda}(b))\setminus B(F_{\lambda}(c))=
\{\tilde{f}_{i_1}\cdots\tilde{f}_{i_l}(u_{\lambda}\otimes b)\ |\
i_1, \cdots, i_l\in I, l\geqslant0\}\setminus\{0\},\] otherwise if
$u_{\lambda}\otimes b\in B(F_{\lambda}(c))$,
$B(F_{\lambda}(b))=B(F_{\lambda}(c))$. If
$\tilde{e}_i(u_{\lambda}\otimes b)=0$ for all $i\in I$, assume that
$u_{\lambda}\otimes b\notin B(F_{\lambda}(b))\setminus
B(F_{\lambda}(c))$. We have $u_{\lambda}\otimes b\in
B(F_{\lambda}(c))$ and it is of the form
$\tilde{f}_{l_1}\cdots\tilde{f}_{l_t}(u_{\lambda}\otimes d)$ for
some $l_1, \cdots, l_t\in I$, $t\geqslant0$ and $d\in B(\mu)$ where
$d\geqslant c>b$. Since $\tilde{e}_i(u_{\lambda}\otimes b)=0$ for
all $i\in I$, it implies $t=0$ and $u_{\lambda}\otimes
b=u_{\lambda}\otimes d$ which is a contradiction. Thus
$u_{\lambda}\otimes b\in B(F_{\lambda}(b))\setminus
B(F_{\lambda}(b_1))$. Conversely, if $\tilde{e}_i(u_{\lambda}\otimes
b)\neq0$ for some $i\in I$, $\tilde{e}_i(u_{\lambda}\otimes
b)=u_{\lambda}\otimes \tilde{e}_ib\neq0$ where
$wt\tilde{e}_ib=wtb+\alpha_i$. It follows that $\tilde{e}_ib>b$ and
furthermore, $\tilde{e}_ib\geqslant c$. Hence
\[u_{\lambda}\otimes b=\tilde{f}_i\tilde{e}_i(u_{\lambda}\otimes b)
=\tilde{f}_i(u_{\lambda}\otimes \tilde{e}_ib)\in
B(F_{\lambda}(c)).\] We have proved the claim which implies the
theorem.
\end{proof}

By deleting superfluous terms in the filtration $(3.1)$, we have a
nice composition series of $V(\lambda)\otimes V(\mu)$.
\begin{corollary}\label{compseries}
For $\lambda$, $\mu\in P_+$, there is a nice ascending composition
series of $U_q(\mathfrak g)$-module $V(\lambda)\otimes V(\mu)$ by
listing the elements in $\{F_{\lambda}(b)\ |\ b\in B(\mu),
\tilde{e}_i(u_{\lambda}\otimes b)=0\ \  \forall i\in I \}$ according
to descending order on $B(\mu)$.
\end{corollary}

Lusztig's conjecture for $\mathfrak g$ of finite type is then an
immediate consequence of the Corollary~\ref{compseries}.

\begin{corollary}\label{conjecture}
For $\lambda$, $\mu\in P_+$ and $\mathfrak g$ of finite type, there
is a nice composition series of $U_q(\mathfrak g)$-module
$V(\lambda)\otimes V(-\mu)$ by listing the elements in
$\{F_{\lambda}(b)\ |\ b\in B(-\mu), \tilde{e}_i(u_{\lambda}\otimes
b)=0\ \  \forall i\in I \}$ according to descending order on
$B(-\mu)$.
\end{corollary}

\begin{example} For $\mathfrak g=sl_3$, consider the $U_q(\mathfrak
g)$-mod $V(\Lambda_1)\otimes V(-\Lambda_1-\Lambda_2)$ as in
Example~\ref{nocom}. Since $V(-\Lambda_1-\Lambda_2)\cong
V(\Lambda_1+\Lambda_2)$ where the total order on the crystal basis
$B(\Lambda_1+\Lambda_2)$ of $V(\Lambda_1+\Lambda_2)$ is given as in
Example~\ref{ordereg}, there exists a nice filtration of the tensor
product
\[0\subseteq F_{\Lambda_1}(\small\Yvcentermath1\young(11,2))\subseteq F_{\Lambda_1}(\small\Yvcentermath1\young(12,2))
\subseteq F_{\Lambda_1}(\small\Yvcentermath1\young(11,3))\subseteq
F_{\Lambda_1}(\small\Yvcentermath1\young(12,3))\subseteq
F_{\Lambda_1}(\small\Yvcentermath1\young(13,2))\]\[\subseteq
F_{\Lambda_1}(\small\Yvcentermath1\young(22,3))\subseteq
F_{\Lambda_1}(\small\Yvcentermath1\young(13,3))\subseteq
F_{\Lambda_1}( \small\Yvcentermath1\young(23,3))=V(\Lambda_1)\otimes
V(-\Lambda_1-\Lambda_2).\] One can check that $u_{\Lambda_1}\otimes
\small\Yvcentermath1\young(11,2),\ u_{\Lambda_1}\otimes
\small\Yvcentermath1\young(12,2), \ u_{\Lambda_1}\otimes
\small\Yvcentermath1\young(12,3)$ are maximal vectors while others
are not. Hence \[0\subsetneq
F_{\Lambda_1}(\small\Yvcentermath1\young(11,2))\subsetneq
F_{\Lambda_1}(\small\Yvcentermath1\young(12,2))
 \subsetneq
F_{\Lambda_1}(\small\Yvcentermath1\young(12,3))=V(\Lambda_1)\otimes
V(-\Lambda_1-\Lambda_2)\] is the nice composition of
$V(\Lambda_1)\otimes V(-\Lambda_1-\Lambda_2)$ where $
F_{\Lambda_1}(\small\Yvcentermath1\young(11,2))\cong
V(2\Lambda_1+\Lambda_2)$,
$F_{\Lambda_1}(\small\Yvcentermath1\young(12,2))/F_{\Lambda_1}(\small\Yvcentermath1\young(11,2))\cong
V(2\Lambda_2)$,
$F_{\Lambda_1}(\small\Yvcentermath1\young(12,3))/F_{\Lambda_1}(\small\Yvcentermath1\young(12,2))\cong
V(\Lambda_1)$.
\end{example}

From the proof of Theorem~\ref{main} one can derive the generalized
Littlewood-Richardson rule  for symmetrizable Kac-Moody algebra
$\mathfrak g$, that is,
\[V(\lambda)\otimes V(\mu)\cong \bigoplus_{b\in B(\mu), \
 \tilde{e}_i(u_{\lambda}\otimes b)=0\ \forall i\in I}V(\lambda+wtb).\]
This generalized Littlewood-Richardson rule was proved by Littelmann
using  path model \cite{Littelmann:1994}, see also
\cite{Kashiwara:1990}. One can see from the tensor rule of crystal
bases that
 $\tilde{e}_i(u_{\lambda}\otimes b)=0$ for all $i\in I$ is equivalent to
 \[\tilde{e}_i^{\langle h_i, \lambda\rangle +1}b=0 \ \textrm{for\ all}\ i\in I\]
 and
 such a crystal basis element $b$ is called \textit{$\lambda$-dominant} in
 \cite{Littelmann:1994}.

\subsection{Comparison With Lusztig's Composition Series}

As stated in the introduction, one can also construct a composition
series of $V(\lambda)\otimes V(\mu)$ inductively in Lusztig's
manner. To be precise, for any $M\in \mathcal O_{int}$ with a
canonical basis, we write $M$ as a direct sum of isotypical
components $M=\bigoplus_{\xi\in P_+}M[\xi].$ Let $\lambda_1$ be a
maximal weight in the set $\{\xi\in P_+ |\ M[\xi]\neq 0\}$. We can see
from the proof of Proposition 27.1.7 in \cite{Lusztig:1993} that
there exists a nice submodule $V_1\cong V(\lambda_1)$ of $M$. Go on
this procedure by changing $M$ to $M_2:= M/V_1$ and so on.
Thus we have a nice $U_q(\mathfrak g)$-submodule $V_i\cong
V(\lambda_i)$ of $M_i$ for some $\lambda_i\in P_+$ maximal in the
weights of $M_i$ where $M_1=M$ and $M_{i+1}=M_i/V_i$. Let $\pi_i$ be
the canonical map $\pi_i: M_i\longrightarrow M_{i+1}.$ We obtain
then a sequence consisting of nice surjective $U_q(\mathfrak
g)$-maps
\[M=M_1\xrightarrow{\ \pi_1\ }\ M_2\xrightarrow{\ \pi_2\ }\cdots\xrightarrow{\pi_{i-1}}M_i
\xrightarrow{\ \pi_i\ }M_{i+1}\xrightarrow{\pi_{i+1}}\cdots .\] We
define $F_i(M)$ to be the kernel of
$\pi_i\circ\pi_{i-1}\circ\cdots\circ\pi_1$ for $i\geqslant 1$ and
set $F_0(M)=0$. One can see easily from the construction that
\[0=F_0(M)\subseteq F_1(M)\subseteq \cdots\subseteq F_i(M)\subseteq F_{i+1}(M)\subseteq\cdots\ \ \ \ (3.2)\]
is a nice composition series of $M$ where $F_i(M)/F_{i-1}(M)\cong
V(\lambda_i)$. Furthermore, it is clear to see that
$\lambda_i\geqslant \lambda_j$ for $i<j$ if they are comparable. In
particular, for $\lambda, \mu\in P_+$, there is a nice composition
series of $V(\lambda)\otimes V(\mu)$. We denote
 by $F_i$ the $U_q(\mathfrak g)$-submodule $F_i(V(\lambda)\otimes V(\mu))$ of $V(\lambda)\otimes V(\mu)$ defined
above for simplicity.\\
\indent Let $b^{\prime}_j$ be the unique highest weight element in
$B(F_j)\setminus B(F_{j-1})$. We know from the previous subsection
that $b^{\prime}_j\in B(\lambda)\otimes B(\mu)$ is of the form
$u_{\lambda}\otimes c_j$ for some $c_j\in B(\mu)$ such that
$\tilde{e}_i(u_{\lambda}\otimes c_j)=0$ for all $i\in I$. One can
see that $\lambda_j=\lambda+wtb_j$ and $\{c_j\ |\ j=1,2,\cdots \}$
is a complete set of elements $b$ such that $u_{\lambda}\otimes b$
is maximal. One can arrange a total order on $B(\mu)$ satisfying the
following two conditions,
\begin{itemize}
\item[(i)]
for $b, c\in B(\mu)$, $b<c$ if $wtb<wtc$.
\item[(ii)]
$c_1>c_2>c_3>\cdots>c_j>c_{j+1}>\cdots$.
\end{itemize}
Indeed we can define $u_{\mu}$ to be the maximum in $B(\mu)$ (one
can see $u_{\mu}=c_1$), then choose an element in
$B(\mu)\setminus\{u_{\mu}\}$ maximal in weight to be the second and
so on only to ensure that $c_1>c_2>c_3>\cdots>c_j>c_{j+1}>\cdots$.
It is feasible since one can see from the inductive construction of
composition series that $wtc_i\geqslant wtc_j$ for $i<j$ if they are
comparable. Once such a total order on $B(\mu)$ is fixed, we
immediately obtain, by Corollary~\ref{compseries}, a nice
composition series of $V(\lambda)\otimes V(\mu)$
\[0\subseteq F_{\lambda}(c_1)\subseteq F_{\lambda}(c_2)\subseteq\cdots\subseteq F_{\lambda}(c_i)\subseteq
F_{\lambda}(c_{i+1})\subseteq\cdots\ \ \ \ (3.3). \] It is clear
that (3.3) coincides with (3.2) when $M=V(\lambda)\otimes V(\mu)$,
i.e.
$F_i=F_{\lambda}(c_i).$\\
\indent Conversely, if we construct the nice composition series of
$V(\lambda)\otimes V(\mu)$
\[0:= F_{\lambda}(b_0)\subseteq F_{\lambda}(b_1)\subseteq
F_{\lambda}(b_2)\subseteq\cdots\subseteq F_{\lambda}(b_i)\subseteq
F_{\lambda}(b_{i+1})\subseteq\cdots\ \ \ \ (3.4) \] as in the
previous subsection, it can be seen from the choice of total order
that $\lambda_i\geqslant\lambda_j$ for $i<j$ if they are comparable
where $\lambda_i\in P_+$ is such that
$F_{\lambda}(b_i)/F_{\lambda}(b_{i-1})\cong V(\lambda_i)$. Hence for
$M=V(\lambda)\otimes V(\mu)=M_1$, we define
$M_i=M/F_{\lambda}(b_{i-1})$,
$V_i=F_{\lambda}(b_i)/F_{\lambda}(b_{i-1})$ and $\pi_i$ as stated
above. It follows easily that the composition series constructed in
Lusztig's manner is exactly (3.4), i.e. $F_i(M):=
ker(\pi_i\circ\pi_{i-1}\circ\cdots\circ\pi_1)=F_{\lambda}(b_i).$
Hence we get the same nice composition series of the tensor product
in two different approaches.

\section{Nice Filtration of $V(\lambda)\otimes V(-\mu)$}

\subsection{Filtration}

In the previous section we have proved, by
Corollary~\ref{conjecture}, Lusztig's conjecture that the
$U_q(\mathfrak g)$-module $V(\lambda)\otimes V(-\mu)$ has a nice
composition series for $\mathfrak g$ of finite type and $\lambda,
\mu\in P_+$. For an arbitrary symmetrizable Kac-Moody algebra
$\mathfrak g$, the $U_q(\mathfrak g)$-module $V(\lambda)\otimes
V(-\mu)$ also admits a canonical basis as mentioned previously. But
the tensor product
 may have infinite dimensional weight spaces (when $\lambda$ and $\mu$
 are both nontrivial) and have no maximal weights. Therefore it does not
belong to  category $\mathcal O_{int}$ and Lusztig's approach to
construct nice submodules of $V(\lambda)\otimes V(-\mu)$ fails while
our method still works in this case. To be precise, though we cannot
obtain a composition series of the tensor product in general,
we find a nice filtration of it instead which helps us to understand the structure of this module.\\
\indent Indeed, we can define a total order on $B(-\mu)$ similarly.
For $b\in B(-\mu)$ which is of the form
$\tilde{e}_{i_1}\cdots\tilde{e}_{i_l}u_{-\mu}$, set $l(b)=l$ and
define $|b|$ to be the $l(b)$-tuple $(i_1, \cdots, i_{l(b)})$ such
that $(i_1, \cdots, i_{l(b)})$ is minimal in lexicographic order
among tuples $(j_1, \cdots, j_{l(b)})$ such that
 $\tilde{e}_{j_1}\cdots\tilde{e}_{j_{l(b)}}u_{-\mu}=b$,
 i.e. 
 \[|b|=min\{ (j_1, \cdots, j_{l(b)})\ |\ b=\tilde{e}_{j_1}
 \cdots\tilde{e}_{j_{l(b)}}u_{-\mu} \}.\]
 Set $|u_{-\mu}|=0$. A total order on $B(-\mu)$ is defined as
 \[b_1\leqslant b_2\ \ \textrm{iff} \ \ l(b_1)<l(b_2)\ \  \textrm{or}\ \ l(b_1)=l(b_2)\ \
 \textrm{but} \ \ |b_1|\leqslant|b_2|. \]
As in section 3, for $b\in B(\mu)$, $V_b(-\mu)$ is defined as a
$k$-subspace of $V(-\mu)$ spanned by all $G(c)$ such that
$c\geqslant b$ and let $F_{\lambda}(b)$ be the $U_q(\mathfrak
g)$-submodule of $V(\lambda)\otimes V(-\mu)$ generated by
$u_{\lambda}\otimes V_b(-\mu)$, i.e. 
\[F_{\lambda}(b):=U_q(\mathfrak
g)(u_{\lambda}\otimes V_b(-\mu)).\] As the proof of
Theorem~\ref{main}, we have the following theorem by
Lemma~\ref{kashiwaralemma}.
\begin{theorem}
For $\lambda$, $\mu\in P_+$, $\{\ F_{\lambda}(b)\ |\ b\in B(-\mu)\}$
forms a nice descending filtration of $V(\lambda)\otimes V(-\mu)$ as
the following
\[V(\lambda)\otimes V(-\mu)=F_{\lambda}(b_1)\supseteq F_{\lambda}(b_2)\supseteq F_{\lambda}(b_3)
\supseteq\cdots\ \ \ \ \ (4.1)\] where $u_{-\mu}=b_1<b_2<b_3<\cdots$
is a complete list of $B(-\mu)$. Moreover, for two neighbors $b<c$
in $B(-\mu)$, $F_{\lambda}(b)/F_{\lambda}(c)\cong V(\lambda+wtb)$ if
$\tilde{e}_i(u_{\lambda}\otimes b)=0$ for all $i\in I$, otherwise
$F_{\lambda}(b)=F_{\lambda}(c)$.
\end{theorem}

Actually the order on $B(-\mu)$ can be chosen only to satisfy the
property that $b_1<b_2$ if $wtb_1<wtb_2$. In contrast to
Corollary~\ref{conjecture}, usually we cannot get a nice composition
series of $V(\lambda)\otimes V(-\mu)$ by deleting superfluous terms
in (4.1). More precisely, the intersection of all submodules in
(4.1) might be nonzero. For example, when $\mathfrak g$ is of affine
type and $\lambda-\mu$ is of a negative level,
$F_{\lambda}(b)=V(\lambda)\otimes V(-\mu)$ for
all $b\in B(-\mu)$.\\
\indent Similarly, with the order on $B(\lambda)$ defined in section
3, we can construct another nice filtration of $V(\lambda)\otimes
V(-\mu)$. For $b\in B(\lambda)$, define $F_{-\mu}(b)$ to be the
$U_q(\mathfrak g)$-submodule of $V(\lambda)\otimes V(-\mu)$
generated by $G(c)\otimes u_{-\mu}$ for all $c\leqslant b$. Note
that when we change $U_q(\mathfrak g)^+$ to $U_q(\mathfrak g)^-$,
Lemma~\ref{kashiwaralemma} is also true which implies the following
theorem.

\begin{theorem}
  For $\lambda$, $\mu\in P_+$, $\{\ F_{-\mu}(b)\ |\ b\in
B(\lambda)\}$ forms a nice descending filtration of
$V(\lambda)\otimes V(-\mu)$ as the following
\[V(\lambda)\otimes V(-\mu)=F_{-\mu}(b_1)\supseteq F_{-\mu}(b_2)\supseteq F_{-\mu}(b_3)
\supseteq\cdots\ \ \ \ \ (4.2)\] where
$u_{\lambda}=b_1>b_2>b_3>\cdots$ is a complete list of $B(\lambda)$.
Moreover, for two neighbors $b>c$ in $B(\lambda)$,
$F_{-\mu}(b)/F_{-\mu}(c)\cong V(-\mu+wtb)$ if $\tilde{f}_i(b\otimes
u_{-\mu})=0$ for all $i\in I$, otherwise, $F_{-\mu}(b)=F_{-\mu}(c)$.
\end{theorem}

\subsection{Affine Type Case}

 For $\lambda\in P$, note that there is a subcrystal $B^{max}(\lambda)$ of
 $B(U_q(\mathfrak g)a_{\lambda})$ consisting of some $*$-extremal elements which is
 exactly the crystal basis of extremal weight module $V^{max}(\lambda)$
 \cite{Kashiwara:1994}. It is proved in  \cite{Kashiwara:1994} that
  \[V^{max}(\lambda)\cong V^{max}(w\lambda)\] for any $w\in W$ and
$V^{max}(\lambda)\cong V(\lambda)$ for $\lambda\in \pm P_{+}$.

\begin{proposition}\label{lengthlim}
(\cite{Kashiwara:1994}) For any connected component $B$ of
$\widetilde{B}$, there is an $l>0$ such that $(wtb,wtb)\leqslant l\
\textrm{for\ all}\ b\in B.$ Moreover, $B$ contains an extremal
vector and can be embedded into $B^{max}(\lambda)$ for some
$\lambda\in P$.
\end{proposition}

For $\mathfrak g$ of affine type, let $c\in\mathfrak h$ be the
canonical central element of $\mathfrak g$. Given $\lambda\in P$, we
define $\langle c,\lambda\rangle$ to be the level of  $\lambda$,
denoted by $level(\lambda)$. It follows immediately from
Proposition~\ref{lengthlim} the following corollary.
\begin{corollary}\label{affine}
\begin{itemize}
\item[(i)]
For $\lambda$ with $level(\lambda)>0$, $B(U_q(\mathfrak
g)a_{\lambda})$ is a union of highest weight crystals.
\item[(ii)]
For $\lambda$ with $level(\lambda)<0$, $B(U_q(\mathfrak
g)a_{\lambda})$ is a union of lowest weight crystals.
\end{itemize}
\end{corollary}

It follows from the corollary that for $\lambda, \mu\in P_+$,
$B(\lambda,-\mu)$ is a union of highest ($resp$. lowest) weight
crystals if $level(\lambda-\mu)>0$ ($resp$. $level(\lambda-\mu)<0$).
We define $W(\lambda,-\mu)$ ($resp.$ $U(\lambda,-\mu)$) to be a
$k$-subspace $\bigcap_{b\in B(-\mu)}F_{\lambda}(b)$ ($resp.$
$\bigcap_{b\in B(\lambda)}F_{-\mu}(b)$) of $V(\lambda)\otimes
V(-\mu)$ and set \[M(\lambda,-\mu)=(V(\lambda)\otimes
V(-\mu))/W(\lambda,-\mu)\] \[(resp.\ \
N(\lambda,-\mu)=(V(\lambda)\otimes V(-\mu))/U(\lambda,-\mu)).\ \ \
\] Denote by $B^{+}(\lambda,-\mu)$ ($resp.$ $B^{-}(\lambda,-\mu)$)
the subcrystal of $B(\lambda, -\mu)$ which is the union of all
connect components of $B(\lambda, -\mu)$ that are not highest
($resp.$ lowest) weight crystals.

\begin{proposition} For $\lambda, \mu\in P_+$,
\begin{itemize}
\item[(i)]
both $W(\lambda,-\mu)$ and $U(\lambda,-\mu)$ are nice $U_q(\mathfrak
g)$-submodules of $V(\lambda)\otimes V(-\mu)$. Moreover,
$B(W(\lambda,-\mu))= B^{+}(\lambda,-\mu)$ and $B(U(\lambda,-\mu))=
B^{-}(\lambda,-\mu)$.
\item[(ii)]
both $M(\lambda,-\mu)$ and $N(\lambda,-\mu)$ admit canonical bases
and $B(M(\lambda,-\mu))= B(\lambda,-\mu)\setminus
B^+(\lambda,-\mu)$, $B(N(\lambda,-\mu))= B(\lambda,-\mu)\setminus
B^-(\lambda,-\mu)$.
\end{itemize}
\end{proposition}
\begin{proof}
$W(\lambda,-\mu)$ admits a $U_q(\mathfrak g)$-action since every
$F_{\lambda}(b)$ does. The conclusion for $W(\lambda,-\mu)$ in
$(\romannumeral 1)$ follows from Theorem~\ref{main} and that any
maximal vector in $B(\lambda, -\mu)$ is of the form
$u_{\lambda}\otimes b$ with $b\in B(-\mu)$ and
$\varepsilon_i(b)\leqslant\langle h_i,\lambda\rangle$ for all $i\in I$.
It is similar for $U(\lambda,-\mu)$ and $(\romannumeral 2)$ is
implied by $(\romannumeral 1)$.
\end{proof}
When $\mathfrak g$ is of finite type, one can see that
$W(\lambda,-\mu)=U(\lambda,-\mu)=0$ and both (3.1) and (4.2) provide
composition series of $V(\lambda)\otimes V(-\mu)$ by deleting
superfluous terms.

For two crystals $B_1$ and $B_2$ where $B_1$ is connected, let
$[B_2: B_1]$ be the cardinality of the set which consists of all
connected components of $B_2$ isomorphic to $B_1$, i.e. $[B_2:
B_1]=\{B\subset B_2\ |\ B\cong B_1\}^{\#}.$
\begin{theorem}\label{multiplicity}
For $\lambda\in P_+$ and $\mu\in P$, $[B(U_q(\mathfrak g)a_{\mu}):
B(\lambda)]=dimV(\lambda)_{\mu}.$
\end{theorem}

\begin{proof}
We only need to find out all maximal vectors in $B(U_q(\mathfrak
g)a_{\mu})$. Note that $B(U_q(\mathfrak g)a_{\mu})=B(\infty)\otimes
T_{\mu}\otimes B(-\infty)$ and $\tilde{e}_i$ acts on it as
\begin{displaymath} \tilde{e}_i(b_1\otimes t_{\mu}\otimes b_2)=
\left\{
\begin{array}{ll}
(\tilde{e}_ib_1)\otimes t_{\mu}\otimes b_2 & \textrm{if
$\varphi_i(b_1)+\langle h_i, \mu\rangle\geqslant \varepsilon_i(b_2)$
}\\

b_1\otimes t_{\mu}\otimes(\tilde{e}_ib_2) & \textrm{if
$\varphi_i(b_1)+\langle h_i, \mu\rangle<\varepsilon_i(b_2)$.}
\end{array} \right.
\end{displaymath}
 Assume that $b_1\otimes t_{\mu}\otimes b_2$ is maximal, since $\tilde{e}_ib_2\neq
 0$ for all $b_2\in B(-\infty)$, we have $\tilde{e}_ib_1=0$ and
 \[\varphi_i(b_1)+\langle h_i, \mu \rangle\geqslant
 \varepsilon_i(b_2)\ \ \ \ \ \ \ \ \  (4.3)\]  for all $i\in I$. Hence
 $b_1=u_{\infty}$ which is the image of $1$. \\

 Now, we claim that
 $u_{\infty}\otimes t_{\mu}\otimes b_2$ is
a maximal vector of weight $\lambda$ iff $wtb_2=\lambda-\mu$ and
$\varphi_i(b_2)\leqslant \langle h_i, \lambda\rangle$ for all $i\in I$.
Indeed, if $u_{\infty}\otimes t_{\mu}\otimes b_2$ is maximal and
$wt(u_{\infty}\otimes t_{\mu}\otimes b_2)=\mu+wtb_2=\lambda$, then
$wtb_2=\lambda-\mu$ and (4.3) holds which can be rewritten as
$\langle h_i, \mu\rangle\geqslant
 \varepsilon_i(b_2)$ since $\varphi_i(u_{\infty})=0$. It follows from
 $\varphi_i(b_2)-\varepsilon_i(b_2)=\langle h_i , wtb_2\rangle$ that
 $\langle h_i, \mu \rangle\geqslant\varphi_i(b_2)-\langle h_i,
 wtb_2\rangle$ which implies $\varphi_i(b_2)\leqslant \langle h_i,
 \lambda\rangle$. The other side of the claim is easy.\\

 It has been shown by Kashiwara in \cite{Kashiwara:1991} that for $\xi\in P_+$
 there is an embedding of crystals \[\tau :\  B(-\xi)\longrightarrow
 T_{-\xi}\otimes B(-\infty)\] whose image is $Im\tau=\{t_{-\xi}\otimes b\ |\
 \varphi_i^{*}(b)\leqslant\langle h_i, \xi\rangle\ \forall\ i\in
 I\}.$ Hence for $\eta\in P$, \[\{ b\in B(-\infty)_{\xi-\eta} |\
 \varphi_i^{*}(b)\leqslant\langle h_i, \xi\rangle\ \forall\ i\in
 I\}^{\#}=dimV(-\xi)_{-\eta}=dimV(\xi)_{\eta}\
 \ \  \ (4.4).\] Recall that $*$ acts bijectively on $B(-\infty)$.
By restricting the $*$-action on $\{b\in B(-\infty)\ |\
 \varphi_i(b)\leqslant\langle h_i, \lambda\rangle\ \forall\ i\in
 I\}$ we get a bijection
$\{b\in B(-\infty)|
 \varphi_i(b)\leqslant\langle h_i, \lambda\rangle\ \forall\ i\in
 I\}\longleftrightarrow \{b\in B(-\infty)|
 \varphi_i^{*}(b)\leqslant\langle h_i, \lambda\rangle\ \forall\ i\in I\}.$
 Hence there is a bijection $\{b\in B(-\infty)_{\lambda-\mu}|
 \varphi_i(b)\leqslant\langle h_i, \lambda\rangle\ \forall\ i\in I\}
 \longleftrightarrow \{b\in B(-\infty)_{\lambda-\mu}|$
 $\varphi_i^{*}(b)\leqslant\langle h_i, \lambda\rangle\ \forall\ i\in
I\}.$
 From (4.4) and the claim above we know that the number of maximal
 vectors in $B(U_q(\mathfrak g)a_{\mu})$ of weight $\lambda$ equals
 \[\{b\in B(-\infty)_{\lambda-\mu}|
 \varphi_i(b)\leqslant\langle h_i, \lambda\rangle\ \forall\
 i\in I\}^{\#}=dimV(\lambda)_{\mu}.\]
 \end{proof}

 Let $P_0$ be the subset of $P_+$ consisting of weights $\lambda$ such that
 $\langle h_i, \lambda\rangle=0$ for all $i\in I$.
 We have the following corollary.

\begin{corollary}
\begin{itemize}
\item[(i)]
$W(\lambda,-\mu)=N(\lambda,-\mu)=0$ and
$M(\lambda,-\mu)=U(\lambda,-\mu)$
$=V(\lambda)\otimes V(-\mu)$ if
$level(\lambda-\mu)>0$.
\item[(ii)]
$W(\lambda,-\mu)=N(\lambda,-\mu)=V(\lambda)\otimes V(-\mu)$ and
$M(\lambda,-\mu)=U(\lambda,-\mu)=0$ if $level(\lambda-\mu)<0$.
\item[(iii)]
$M(\lambda,-\mu)=N(\lambda,-\mu)$ is a 1-dimensional trivial module
if $\lambda-\mu\in P_0$, otherwise if $\lambda-\mu\notin P_0$ is of
level 0, $W(\lambda,-\mu)=U(\lambda,-\mu)=V(\lambda)\otimes V(-\mu)$
and $M(\lambda,-\mu)=N(\lambda,-\mu)=0$.
\end{itemize}
\end{corollary}
\begin{proof}
$(\romannumeral 1)$,  $(\romannumeral 2)$ come from
Corollary~\ref{affine}. $(\romannumeral 3)$ holds since there is no
highest or lowest weight subcrystal in $B(\lambda, -\mu)$ if
$\lambda-\mu\notin P_0$ is of level 0 while there is only one
trivial subcrystal for $\lambda-\mu\in P_0$ by
Theorem~\ref{multiplicity}.
\end{proof}

We can see from this corollary that for $\mathfrak g$ of affine
type, (4.1) ($resp.$ (4.2)) provides a nice composition series of
$V(\lambda)\otimes V(-\mu)$ by deleting superfluous terms when
$\lambda-\mu$ is of a positive ($resp.$ negative) level.


\begin{thebibliography}{00}


\bibitem{Beck:2004}
J. Beck and H. Nakajima, \textit{Crystal bases and two-sided cells
of quantum affine algebras}, Duke Math. J. Volume 123, Number 2
(2004), 335-402.

\bibitem{FJKMMT:2004}
B. Feigin, M. Jimbo, M. Kashiwara, T. Miwa, E. Mukhin, Y. Takeyama,
\textit{A functional model for the tensor product of level 1 highest
and level $-1$ lowest modules for the quantum affine algebra
$U_{q}(\widehat{sl_{2}}$)}, Euro. J. Comb. Vol. 25. Issue 8 (2004),
1197-1229.

\bibitem{Jantzen:1996}
Jens. Carsten Jantzen, \textit{Lectures on quantum groups}, Gradu.
Stud. in. Math. Vol. 6. (AMS, 1996).

\bibitem{Kashiwara:1990}
M. Kashiwara, \textit{Crystalizing the q-analogue of universal
enveloping algebras}, Comm. Math. Phys. 133 (1990), 249-260.

\bibitem {Kashiwara:1991}
M. Kashiwara, \textit{On crystal bases of the q-analogue of
universal enveloping algebras}, Duke. Math. J. 63 (1991), 465-516.

\bibitem{Kashiwara:1993}
M. Kashiwara, \textit{The crystal bases and Littelmann's refined
Demazure character formula}, Duke. Math. J. Vol. 71. No.3 (1993),
839-858.

\bibitem{Kashiwara2:1993}
M. Kashiwara, \textit{Global crystal bases of quantum groups}, Duke.
Math. J. Vol. 69. No. 2  (1993), 455-485.

\bibitem{Kashiwara:1994}
M. Kashiwara, \textit{Crystal bases of modified enveloping algebra},
Duke. Math. J. Vol. 73. No. 2 (1994), 383-413.

\bibitem{Littelmann:1994}
P. Littelmann, \textit{A Littlewood-Richardson rule for
symmetrizable Kac-Moody algebras}, Invent. Math. 116 (1994),
329-346.

\bibitem{Lusztig:1990}
G. Lusztig, \textit{Canonical bases arising from quantized
enveloping algebras}, J. Amer. Math. Soc., 3 (1990), 447-498.

\bibitem{Lusztig2:1990}
G. Lusztig, \textit{Canonical bases arising from quantized
enveloping algebras
 II}, Progr. Theor. Phys. Suppl., 102 (1990), 175-201.


\bibitem{Lusztig:1991}
G. Lusztig, \textit{Quivers, perverse sheaves and quantized
enveloping algebras}, J. Amer. Math. Soc., 4 (1991), 365-421.


\bibitem{Lusztig:1992}
G. Lusztig, \textit{Canonical bases in tensor Product}, Poc. Nat.
Acad. Sci. USA. 89 (1992), 8177-8179.

\bibitem{Lusztig:1993}
G. Lusztig, \textit{Introduction to quantum groups},
Birkh$\ddot{a}$user, Boston, 1993.

\bibitem{Lusztig:1995}
G. Lusztig, \textit{Quantum groups at v=infinity}, in "Functional
analysis on the eve of the 21st century", vol.I, Progr.in Math.131,
Birkh$\ddot{a}$user, Boston 1995, 199-221.

\bibitem{Malkin:2003}
A. Malkin, \textit{Tensor product varieties and crystals: The ADE
case}, Duke Math. J. Volume 116, Number 3 (2003), 477-524.

\bibitem{Nakajima:2001}
H. Nakajima, \textit{Quiver varieties and tensor products}, Invent.
Math. 146 (2001), 399--449. $\backslash$CMP1 865 400




\end{thebibliography}
\end{document}